\documentclass{article}
\newcommand{\Addresses}{{
		\bigskip
		\footnotesize
		
		\textsc{Mathematical Institute, University of Oxford, Oxford, OX2 6GG, UK}\par\nopagebreak
		\textit{E-mail address:} \texttt{gorodetsky@maths.ox.ac.uk}
}}
\author{Ofir Gorodetsky} \title{Magic squares, the symmetric group and M\"obius randomness}
\date{}
\usepackage[margin=1in]{geometry}
\usepackage[all]{xy}
\usepackage{amssymb}
\usepackage{amsfonts}
\usepackage{amsthm}
\usepackage{amsmath}
\usepackage{mathtools}
\usepackage{hyperref}

\theoremstyle{plain}
\newtheorem{thm}{Theorem}[section]
\newtheorem{lem}[thm]{Lemma}  
\newtheorem{proposition}[thm]{Proposition}
\newtheorem{conj}[thm]{Conjecture}
\newtheorem{cor}[thm]{Corollary}

\theoremstyle{remark}
\newtheorem{remark}[thm]{Remark}

\newcommand{\cupdot}{\mathbin{\mathaccent\cdot\cup}}

\newcommand{\ZZ}{\mathbb{Z}}
\newcommand{\FF}{\mathbb{F}}

\newcommand{\RR}{\mathbb{R}}

\newcommand{\Sc}{\mathrm{Sc}}
\newcommand{\Sym}{\mathrm{Sym}}
\newcommand{\Tr}{\mathrm{Tr}}
\newcommand{\sgn}{\mathrm{sgn}}
\mathtoolsset{showonlyrefs}

\numberwithin{equation}{section}

\makeatletter
\newcommand{\subjclass}[2][2020]{%
	\let\@oldtitle\@title%
	\gdef\@title{\@oldtitle\footnotetext{#1 \emph{Mathematics subject classification:} #2}}%
}
\newcommand{\keywords}[1]{%
	\let\@@oldtitle\@title%
	\gdef\@title{\@@oldtitle\footnotetext{\emph{Key words:} #1.}}%
}
\makeatother

\begin{document}
\maketitle
\begin{abstract}
Diaconis and Gamburd computed moments of secular coefficients in the CUE ensemble. We use the characteristic map to give a new combinatorial proof of their result. We also extend their computation to moments of traces of symmetric powers, where the same result holds but in a wider range.

Our combinatorial proof is inspired by gcd matrices, as used by Vaughan and Wooley and by Granville and Soundararajan. We use these CUE computations to suggest a conjecture about moments of characters sums twisted by the Liouville (or by the M\"obius) function, and establish a version of it in function fields.

The moral of our conjecture (and its verification in function fields) is that the Steinhaus random multiplicative function is a good model for the Liouville (or for the M\"obius) function twisted by a random Dirichlet character.

We also evaluate moments of secular coefficients and traces of symmetric powers, without any condition on the size of the matrix. As an application we give a new formula for a matrix integral that was considered by Keating, Rodgers, Roditty-Gershon and Rudnick in their study of the $k$-fold divisor function.
\end{abstract}
\section{Introduction}
Consider the complex unitary group $U(N)$ endowed with the probability Haar measure. The $n$th secular coefficient of $U\in U(N)$ is defined through the expansion
\[ \det(zI+U) = \sum_{n=0}^N z^{N-n} \Sc_n(U). \]

If $A=(a_{i,j})$ is an $m\times n$ matrix with nonnegative integer entries, Diaconis and Gamburd \cite{diaconis2004} define the row-sum vector $\mathrm{row}(A)\in \ZZ^m$ and column-sum vector $\mathrm{col}(A) \in \ZZ^n$ by 
\[ \mathrm{row}(A)_i = \sum_{j=1}^{n} a_{i,j}, \qquad \mathrm{col}(A)_j = \sum_{i=1}^{m} a_{i,j}.\]
Given two partitions $\mu = (\mu_1,\ldots,\mu_m)$ and $\widetilde{\mu} = (\widetilde{\mu}_1,\ldots, \widetilde{\mu}_n)$ they denote by $N_{\mu, \widetilde{\mu}}$ the number of nonnegative $m \times n$ integer matrices $A$ with $\mathrm{row}(A)=\mu$ and $\mathrm{col}(A) = \widetilde{\mu}$. When $m=n=k$ and $\mu_1=\ldots=\mu_k=\widetilde{\mu}_1=\ldots=\widetilde{\mu}_k$, matrices counted by $N_{\mu,\widetilde{\mu}}$ are known as \textit{magic squares of order $k$}, see \cite[\S2.2]{diaconis2004} for a review. Given sequences $(a_1,\ldots,a_{\ell})$ and $(b_1,\ldots,b_{\ell})$ of nonnegative integers, they proved the following equality \cite[Thm.~2]{diaconis2004}:
\begin{equation}\label{eq:DG}
\int_{U(N)} \prod_{j=1}^{\ell} \Sc_j(U)^{a_j} \overline{\Sc_j(U)^{b_j}}\, dU  = N_{\mu, \widetilde{\mu}}
\end{equation}
as long as $\max\{\sum_{j=1}^{\ell} ja_j, \sum_{j=1}^{\ell} jb_j\} \le N$, where $\mu$ and $\widetilde{\mu}$ are the partitions with $a_j$ and $b_j$ parts of size $j$, respectively.

Identity \eqref{eq:DG} answered a question raised in \cite{haake1,sommers2}, where it was shown that $\int_{U(N)}\Sc_n(U)dU = 0$ and $\int_{U(N)}|\Sc_n(U)|^2dU = 1$ hold for $1 \le n \le N$. The results in \cite{diaconis2004} inspired the study of pseudomoments of the Riemann zeta function \cite{conrey2006} and were used in \cite{keating2018} to study the variance of the $k$-fold divisor function in short intervals. Recently, Najnudel, Paquette and Simm studied the distribution of $\Sc_n$ with $n$ growing with $N$ \cite{Najnudel}.

In \S\ref{sec:dg} we give a new combinatorial proof of \eqref{eq:DG}, which makes use of the characteristic map.  This is in the spirit of Bump's derivation \cite[Prop.~40.4]{bumplie} of the Diaconis--Shahshahani moment computation \cite{diaconis1994}.

In \S\ref{sec:symmetricpowers} we show that a result similar to \eqref{eq:DG} holds for traces of symmetric powers in place of secular coefficients, with substantially relaxed conditions. These traces are also the complete homogeneous symmetric polynomials $h_n$ evaluated on the eigenvalues of the matrix. This result can be derived from a theorem of Baxter \cite[Prop.~2.11]{johansson1997} but again, our proof is combinatorial in nature.

In \S\ref{sec:magic} we give two evaluations of \eqref{eq:DG} without any restriction on $N$. One evaluation uses the RSK correspondence and generalizes a result of Rains \cite{Rains}, and the second evaluation uses Gelfand--Tsetlin patterns and generalizes an argument of Rodgers \cite{Rodgers}. These evaluations extend to moments of traces of symmetric powers. As an application, we give a new formula for a matrix integral that was considered by Keating, Rodgers, Roditty-Gershon and Rudnick in their study of the $k$-fold divisor function \cite{keating2018}.

In \S\ref{sec:NT} we show how an analogue of our proof of \eqref{eq:DG} has appeared in number-theoretic works of Vaughan and Wooley and of Granville and Soundararajan. The number-theoretic analogues of \eqref{eq:DG} concern moments of character sums on which there are several unconditional results in the literature. However, our result on symmetric traces corresponds to a conjecture we make about moments of sums of the Liouville (or the M\"obius) function twisted by a Dirichlet character, which seems very difficult and reflects the random nature of M\"obius, see Conjecture \ref{conj:mob}. As we explain in \S\ref{sec:NT}, the conjecture suggests that the Steinhaus random multiplicative function is a good model for the Liouville (or the M\"obius) function twisted by a random Dirichlet character: $\lambda \cdot \chi$ where $\chi$ is chosen uniformly at random from the group of Dirichlet characters modulo $q$ ($q \to \infty$) and $\lambda(n)= (-1)^{\sum_{p,\,k\ge 1: p^k \mid n} 1}$.
\section{Proof of the Diaconis--Gamburd Theorem}\label{sec:dg}
\subsection{The symmetric group}
For a permutation $\pi$ we say that $S$ is an invariant set for $\pi$ if $\pi(S)=S$. Equivalently, $S$ is a union of cycles of $\pi$. Given a sequence $\lambda=(\lambda_1,\ldots,\lambda_{\ell})$ of nonnegative integers that sum to $n$, we define the following function on the symmetric group $S_{n}$ acting on $[n]:=\{1,2,\ldots,n\}$:
\begin{equation}
d_{\lambda}(\pi) = \#\{( A_1,\ldots, A_{\ell}) : \cupdot_i A_i = [n],\, \text{each }A_i\text{ is an invariant set with }  |A_i|=\lambda_i\},
\end{equation}
where $\cupdot$ means disjoint union. We use the letter $d$ here as short for \emph{divisor}, as invariant sets for $\pi$ are analogous to divisors of an integer $m$, and $d_{\lambda}$ is analogous to a generalized divisor function over the integers, with divisors localized at certain scales (in the integers we might define $m\mapsto \#\{ (m_1,m_2,\cdots,m_{\ell}): m_1\cdots m_{\ell}=m,\, \log m_i \in [\lambda_i,\lambda_i+1)\}$). The simplest examples are $d_{(n)}$ which is identically $1$, and $d_{(a,n-a)}$ which counts invariant sets of size $a$. 
\begin{remark}
The function $d \colon S_n \to \mathbb{C}$ given by $d:= \sum_{a=0}^{n} d_{(a,n-a)}$ equals the number of invariant sets, namely $d(\pi)=2^{C(\pi)}$ where $C(\pi)$ is the number of cycles of $\pi$; this is the permutation analogue of the divisor function. 
\end{remark}
\begin{remark}
See \cite{EFG} for a recent application of the analogy between invariant sets for a permutation and divisors of an integer.
\end{remark}
Given sequences $\mu$ and $\widetilde{\mu}$ of nonnegative integers summing to $n$, let us define
\begin{equation}\label{eq:form fnr}
	N_{\mu,\widetilde{\mu}}' := \frac{1}{|S_{n}|}\sum_{\pi \in S_{n}} d_{\mu}(\pi)d_{\widetilde{\mu}}(\pi).
\end{equation}

\begin{proposition}\label{prop:avoid knuth}
Suppose $\mu,\widetilde{\mu}\vdash n$. We have $N_{\mu,\widetilde{\mu}}' =N_{\mu,\widetilde{\mu}}$.
\end{proposition}
\begin{proof}
By definition, given a partition $\lambda= (\lambda_1,\ldots,\lambda_{\ell})\vdash n$ we may express $d_{\lambda}(\pi)$ as a sum over ordered set partitions:
	\begin{equation}\label{eq:alpha pi ais}
	d_{\lambda}(\pi)=\sum_{\substack{(A_1,\ldots,A_\ell): \cupdot A_i = [n]\\ |A_i| = \lambda_i}} \alpha_{A_1,\ldots,A_{\ell}}(\pi)
	\end{equation}
where $\alpha_{A_1,\ldots,A_{\ell}}$ is the indicator function of permutations $\pi\in S_n$ with $\pi(A_i)=A_i$ for all $i$. Applying \eqref{eq:alpha pi ais} with $\lambda=\mu$ and multiplying by \eqref{eq:alpha pi ais} with $\lambda =\widetilde{\mu}$ we obtain
	\begin{equation}
 d_{\mu}(\pi)d_{\widetilde{\mu}}(\pi)=\sum_{\substack{(A_1,\ldots,A_{\ell(\mu)})\\ \cupdot A_i = [n] \\ |A_i| = \mu_i}} \sum_{\substack{(B_1,\ldots,B_{\ell(\widetilde{\mu})})\\ \cupdot B_i = [n]\\ |B_i| = \widetilde{\mu}_i}} \alpha_{A_1,\ldots, A_{\ell(\mu)}}(\pi)  \alpha_{B_1,\ldots, B_{\ell(\widetilde{\mu})}}(\pi)
	\end{equation}
where $\ell(\lambda)$ is the number of parts in a partition. Averaging this over $S_n$ and interchanging the order of summation, we find
 	\begin{equation}\label{eq:inter}
 	N_{\mu,\widetilde{\mu}}' = \frac{1}{n!} \sum_{\substack{(A_1,\ldots,A_{\ell(\mu)})\\ \cupdot A_i = [n]\\ |A_i| = \mu_i}} \sum_{\substack{(B_1,\ldots,B_{\ell(\widetilde{\mu})})\\ \cupdot B_i = [n]\\ |B_i| = \widetilde{\mu}_i}} \sum_{\pi \in S_{n}} \alpha_{A_1,\ldots, A_{\ell(\mu)}}(\pi)  \alpha_{B_1,\ldots, B_{\ell(\widetilde{\mu})}}(\pi).
 \end{equation}
The inner sum in the right-hand side of \eqref{eq:inter} counts permutations $\pi\in S_n$ for which $A_i$ are invariant sets, as well as the $B_j$. In particular $\pi(A_i \cap B_j) \subseteq A_i, B_j$, forcing $\pi(A_i \cap B_j) = A_i \cap B_j$. Conversely, given a permutation such that $\pi(A_i \cap B_j) = A_i\cap B_j$ for all $i$ and $j$, it necessarily satisfies $\pi(A_i)=A_i$ and $\pi(B_j)=B_j$ for all $i$ and $j$. Thus, the inner sum counts $\pi$s with $\pi(A_i \cap B_j) = A_i \cap B_j$. The sets $A_i \cap B_j$ ($1 \le i \le \ell(\mu)$, $1 \le j \le \ell(\widetilde{\mu})$) are pairwise disjoint and their union is $[n]$, and so such $\pi$s are determined uniquely by their restrictions to $A_i \cap B_j$, which may be arbitrary, proving that the inner sum is $\prod_{i,j=1}^{n} |A_i \cap B_j|!$. Hence,
	\begin{equation}\label{eq:sum with magic}
	 	N_{\mu,\widetilde{\mu}}' = \frac{1}{n!} \sum_{\substack{(A_1,\ldots,A_{\ell(\mu)})\\ \cupdot A_i = [n]\\ |A_i| = \mu_i}} \sum_{\substack{(B_1,\ldots,B_{\ell(\widetilde{\mu})})\\ \cupdot B_i =[n] \\  |B_i| = \widetilde{\mu}_i}} \prod_{i,j} |A_i \cap B_j|!.
	\end{equation}
	Observe that the $n \times m$ matrix $C=(|A_i \cap B_j|)$ has $\mathrm{row}(C) =\mu$ and $\mathrm{col}(C) = \widetilde{\mu}$. Hence
	\begin{equation}
		 	N_{\mu,\widetilde{\mu}}'  = \frac{1}{n!}\sum_{\substack{C=(c_{i,j}) \text{ a matrix }\\\text{counted by }N_{\mu,\widetilde{\mu}}}} \prod_{i,j} c_{i,j}! \cdot \# \{ [n] = \cupdot_{i,j}C_{i,j}, \, |C_{i,j}| = c_{i,j}\}.
	\end{equation}
The inner expression in the right-hand side is the
 number of ordered set partitions of $[n]$ into subsets $C_{i,j}$ of size $c_{i,j}$ (these sets correspond to $A_i \cap B_j$ and one reconstructs $A_i$ by $A_i = \cup_{j} C_{i,j}$ and similarly $B_j = \cup_{i} C_{i,j}$). This is just the
 multinomial \[\binom{n}{(c_{i,j}): 1 \le i \le \ell(\mu),\, 1 \le j \le \ell(\widetilde{\mu})} = \frac{n!}{\prod_{i,j} c_{i,j}!},\]
so that \eqref{eq:sum with magic} simplifies to
	\begin{equation}
	N_{\mu,\widetilde{\mu}}'  = \sum_{\substack{C \text{ a matrix }\\\text{counted by }N_{\mu,\widetilde{\mu}}}} 1 = N_{\mu,\widetilde{\mu}} 
\end{equation}
	as claimed.
\end{proof}
In the simple case $\widetilde{\mu}=(n)$ , Proposition~\ref{prop:avoid knuth} reduces to $\sum_{\pi \in S_n} d_{\mu}(\pi)/|S_n|=1$.
\subsection{The characteristic map}
Endow $S_n$ with the uniform probability measure. The characteristic (or Frobenius) map  $\mathrm{Ch}^{(N)}$ is a linear map from class functions on $S_n$ to class functions on $U(N)$, with the property that if $n \le N$ then it is an isometry with respect to the $L_2$-norm, see \cite[Thm.~40.1]{bumplie}. It may be given by
\begin{equation}
\mathrm{Ch}^{(N)}(f) = \frac{1}{n!} \sum_{ \pi \in S_n} f(\pi) p_{\lambda(\pi)},
\end{equation}
see \cite[Thm.~39.1]{bumplie}. Here $\lambda(\pi)$ is the partition associated with $\pi$ (the nondecreasing sequence, summing to $n$, of integers corresponding to the cycle sizes in $\pi$), and $p_{\lambda}$ is the power sum symmetric polynomial associated with $\lambda$, evaluated at the eigenvalues of $U \in U(N)$.
\begin{lem}\label{lem:ch}
Suppose $\lambda \vdash n$. We have
\begin{equation}
	\mathrm{Ch}^{(N)}(\sgn \cdot d_{\lambda}) = e_{\lambda},
\end{equation}	
where $\sgn$ is the sign representation and $e_{\lambda}$ is the elementary symmetric polynomial associated with the partition $\lambda$.
\end{lem}
\begin{proof}
Given $\pi \in S_{n}$, we set $p_{\pi} = p_{\lambda(\pi)}$. We denote by $\ell(\lambda)$ the number of parts in $\lambda$. We then have, by plugging \eqref{eq:alpha pi ais} in the definition of $	\mathrm{Ch}^{(N)}(\sgn \cdot d_{\lambda})$ and interchanging order of summation,
\begin{equation}
\mathrm{Ch}^{(N)}(\sgn \cdot d_{\lambda}) = \frac{1}{n!} \sum_{\substack{(A_1,\ldots,A_{\ell(\lambda)}): \cupdot A_i = [n]\\  \forall i:\, |A_i| = \lambda_i}} \sum_{\substack{\pi \in S_{n}\\   \forall i:\,\pi(A_i)=A_i}}\sgn(\pi)  p_{\pi}.
\end{equation}
We claim that the inner sum is $e_{\lambda}$. Indeed, since $\pi$ is determined by the restrictions $\pi|_{A_i}$, and since $p_{\lambda} = \prod_i p_{\lambda_i}$, we have 
\begin{equation}
\sum_{\substack{\pi \in S_{n}\\  \forall i:\,\pi(A_i)=A_i}} \sgn(\pi)p_{\pi} = \prod_{i=1}^{\ell(\lambda)} \left( \sum_{\pi_i \in S_{A_i}} \sgn(\pi) p_{\pi_i} \right) = \prod_{i=1}^{\ell(\lambda)} \lambda_i! e_{\lambda_i} ,
\end{equation}
where the last equality follows from the Newton--Girard identity $\sum_{\pi \in S_m} \sgn(\pi)p_{\pi}/m! = e_{m}$. 
To finish, note that the number of ordered set partitions of $[n]$ into $\ell(\lambda)$ sets of sizes $(\lambda_i)_{i=1}^{\ell(\lambda)}$ is exactly the binomial coefficient $\binom{n}{\lambda_1,\ldots, \lambda_{\ell(\lambda)}}$.
\end{proof}
\subsection{Conclusion of proof}
Here we establish \eqref{eq:DG}. Let $(a_1,\ldots,a_{\ell})$ and $(b_1,\ldots,b_{\ell})$ be sequences of nonnegative integers satisfying the condition $\max\{\sum_{j=1}^{\ell} ja_j, \sum_{j=1}^{\ell} jb_j\} \le N$. Let $\mu$ and $\widetilde{\mu}$ be the partitions with $a_j$ and $b_j$ parts of size $j$, respectively. 

If $\sum_j j a_j \neq \sum_j j b_j$, it is easy to see that both sides of \eqref{eq:DG} vanish. Indeed, for the right-hand side, note that the integrand is an homogeneous polynomial in the eigenvalues fof $U$, whose degree is nonzero, so its integral must vanish by translation-invariance of the Haar measure. On the other hand, if $N_{\mu,\widetilde{\mu}}$ is nonzero, we must have that $\mu$ and $\widetilde{\mu}$ sum to the same number (if $A=(a_{i,j})$ is a matrix counted by $	N_{\mu,\widetilde{\mu}}$ then both $\mu$ and $\widetilde{\mu}$ sum to $\sum_{i,j} a_{i,j})$.

Now assume $\sum_j j a_j = \sum_j j b_j =n \le N$. As $\prod_{j}\Sc_j(U)^{a_j} \overline{\Sc_j(U)^{b_j}} = {e}_{\mu} \overline{{e}_{\widetilde{\mu}}}$ by definition, the fact that $\mathrm{Ch}^{(N)}$ is an isometry if $n \le N$ shows, through Lemma~\ref{lem:ch}, that the integral in \eqref{eq:DG} is equal to 
\begin{equation}\label{eq:2}
 \frac{1}{|S_{n}|}\sum_{\pi \in S_{n}} (\sgn \cdot d_{\mu})(\pi)\overline{\sgn \cdot d_{\widetilde{\mu}}}(\pi) =  \frac{1}{|S_{n}|}\sum_{\pi \in S_{n}} d_{\mu}(\pi) d_{\widetilde{\mu}}(\pi)=N_{\mu,\widetilde{\mu}}',
\end{equation}
and the proof is concluded by applying Proposition~\ref{prop:avoid knuth}.
\section{Symmetric powers}\label{sec:symmetricpowers}
Let $\Tr \Sym^n(U)$ be the trace of the $n$th symmetric power of $U \in U(N)$. This is also the $n$th complete homogeneous symmetric polynomial $h_n$ evaluated on the eigenvalues of $U$.
\begin{lem}\label{lem:trsym}
Let $( a_j)_{j=1}^{\ell}$, $(b_j)_{j=1}^{\ell}$ be sequences of nonnegative integers. We have
\begin{equation}\label{eq:trsym}
\int_{U(N)} \prod_{j=1}^{\ell} (\Tr \Sym^j(U))^{a_j} \overline{(\Tr \Sym^j(U))^{b_j}}\, dU  = N_{\mu, \widetilde{\mu}}
\end{equation}
as long as $\min \{\sum_{j=1}^{\ell} a_j,\sum_{j=1}^{\ell} b_j\} \le N$, where $\mu$ and $\widetilde{\mu}$ are the partitions with $a_j$ and $b_j$ parts of size $j$, respectively. 
\end{lem}
We start with the following corollary of Lemma~\ref{lem:ch}.
\begin{cor}\label{cor:inv}
	Suppose $\lambda \vdash n$. We have $		\mathrm{Ch}^{(N)}( d_{\lambda}) = h_{\lambda}$.
\end{cor}
\begin{proof}
This follows from Lemma~\ref{lem:ch} through the existence of an involution $\iota$ on the space of symmetric polynomials, with the properties $\iota(\mathrm{Ch}^{(N)}(f))=\mathrm{Ch}^{(N)}(\sgn \cdot f)  $ \cite[Thm.~39.3]{bumplie} and $\iota(e_{\lambda}) = h_{\lambda}$ \cite[Thm.~36.3]{bumplie}. Alternatively, one may repeat the proof of Lemma~\ref{lem:ch} with the Newton--Girard identity $\sum_{\pi \in S_m} p_{\pi}/m! = h_{m}$.
\end{proof}
At this point we can deduce \eqref{eq:trsym} in the restricted range $\max\{\sum_{j=1}^{\ell} ja_j$, $\sum_{j=1}^{\ell} jb_j \}\le N$, in the same way we proved \eqref{eq:DG}.

Next we prove the following well-known identity, often proved as a consequence of the RSK correspondence. Recall that for given partitions $\lambda$ and $\mu$, the Kostka number $K_{\lambda,\mu}$ is defined as the number of semistandard Young tableaux (SSYTs) of shape $\lambda$ and weight $\mu$ \cite[\S2.3]{diaconis2004}.
\begin{lem}
Given $\mu,\widetilde{\mu} \vdash n$ we have
\begin{equation}\label{eq:kostkaiden}
 \sum_{\lambda \vdash n} K_{\lambda,\mu}K_{\lambda,\widetilde{\mu}} = N_{\mu,\widetilde{\mu}}.
\end{equation}
\end{lem}
\begin{proof}
We may expand $e_{\mu}$ and $e_{\widetilde{\mu}}$ in the Schur basis, see \cite[p.~335]{stanley}:
\begin{equation}\label{eq:exps}
 e_{\mu} = \sum_{\lambda\vdash n} K_{\lambda', \mu} s_{\lambda}, \qquad e_{\widetilde{\mu}} = \sum_{\lambda\vdash n} K_{\lambda', \widetilde{\mu}} s_{\lambda},
\end{equation}
where $\lambda'$ is the conjugate of $\lambda$. Let $a_j$ and $b_j$ be the number of $j$s in $\mu$ and $\widetilde{\mu}$, respectively. Orthogonality of Schur functions \cite[Eq.~(22)]{diaconis2004} implies that
\[	\int_{U(n)} \prod_{j=1}^{\ell} \Sc_j(U)^{a_j} \overline{\Sc_j(U)^{b_j}}\, dU  = \sum_{\lambda \vdash n} K_{\lambda',\mu}K_{\lambda',\widetilde{\mu}}=\sum_{\lambda \vdash n} K_{\lambda,\mu}K_{\lambda,\widetilde{\mu}}.\]
On the other hand, this integral was shown to equal $N_{\mu,\widetilde{\mu}}$ in \eqref{eq:DG}.
\end{proof}
We now prove Lemma~\ref{lem:trsym}. 
\begin{proof}
The case $\sum_{j=1}^{\ell} ja_j\neq \sum_{j=1}^{\ell} jb_j$ is treated as in the secular coefficients case. Next, assume that $\sum_j ja_j=  \sum_j jb_j =n$ and $\min \{\sum_{j=1}^{\ell} a_j,\sum_{j=1}^{\ell} b_j\} \le N$. The multiset of eigenvalues of $\Tr \Sym^j (U)$ consists of products of $j$ eigenvalues of $U$, and so the integrand in the left-hand side of \eqref{eq:trsym} is $h_{\mu}\overline{h_{\widetilde{\mu}}}$. We may expand $h_{\lambda}$ in the Schur basis, see Stanley \cite[Cor.~7.12.4]{stanley}:
\[ h_{\mu} = \sum_{\lambda\vdash n} K_{\lambda, \mu} s_{\lambda}.\]
Orthogonality of Schur functions implies that the left-hand side of \eqref{eq:trsym} is 
\begin{equation}\label{eq:partialsum}
\sum_{\substack{\lambda \vdash n\\\ell(\lambda)\le N}} K_{\lambda, \mu} K_{\lambda, \widetilde{\mu}}.
\end{equation}
We claim $K_{\lambda,\mu} \neq 0$ implies $\ell(\lambda) \le \ell(\mu)$ (see e.g.~\cite[Prop.~7.10.5]{stanley}). This follows from the definition of Kostka numbers: the first column of an SSYT counted by $K_{\lambda,\mu}$ contains $\ell(\lambda)$ increasing positive integers less than or equal to $\ell(\mu)$, hence the implication. 

As $\min \{\ell(\mu),\ell(\widetilde{\mu})\} = \min \{ \sum_j a_j, \sum_j b_j \} \le N$ by assumption, we deduce \eqref{eq:partialsum} is equal to the full sum $\sum_{\lambda \vdash n} K_{\lambda, \mu} K_{\lambda, \widetilde{\mu}}$ and the proof is concluded by \eqref{eq:kostkaiden}.
\end{proof}
\begin{remark}
Lemma~\ref{lem:trsym} may also be derived from a  theorem of Baxter on Toeplitz determinants for certain generating functions \cite{Baxter1961} (cf.~\cite[Prop.~2.11]{johansson1997}), special cases of which appeared in earlier works of Szeg\H{o} and Onsager.  Concretely, Baxter proved that (originally in the language of Toeplitz determinants) if $\min\{\ell,m\} \le N$ then
\begin{equation}\label{eq:Baxter}
\int_{U(N)} \frac{1}{\prod_{j=1}^{\ell}\det(I-a_j U)}\frac{1}{\prod_{i=1}^{m}\det(I-b_i \overline{U})}dU = \prod_{i=1}^{m} \prod_{j=1}^{\ell}\frac{1}{1-a_j b_i}
\end{equation}
for complex $|a_j|<1$ and $|b_i|<1$.
Expanding the rational functions in both sides as power series and comparing coefficients, one obtains Lemma~\ref{lem:trsym}.
In a sense, the appearance of magic squares in random matrix theory could have been anticipated due to \eqref{eq:Baxter}.
\end{remark}
\begin{remark}
A weaker version of Lemma~\ref{lem:trsym}, with $\max$ in place of $\min$, may be derived from formulas for averages of ratios of characteristic polynomials \cite{conrey2005howe,bump}.
\end{remark}
\begin{remark}
See \cite[Thm.~1.6]{Najnudel} for a variant of Lemma~\ref{lem:trsym} and \eqref{eq:Baxter}, where one works in the circular $\beta$-ensemble and takes $N \to \infty$.
\end{remark}
\section{General moments}\label{sec:magic}
Given a matrix with nonnegative integer entries, an \textit{SE-chain} is a sequence of entries in which each entry is located weakly to the right of and weakly below the preceding entry. The \textit{length} of an SE-chain is defined as the sum of elements in it. An \textit{ne-chain} is a sequence of nonzero entries in which each entry is  strictly to the right of and strictly above the preceding entry. The \textit{length} of an ne-chain is defined as the number of  elements in it.

The RSK correspondence is a bijection from the set of matrices with nonnegative integers to the set $\{ (P_1,P_2): P_i \text{ are SSYTs with the same shape}\}$, see \cite[Ch.~7.11]{stanley} for its description. The bijection takes a matrix $A$ to a pair $(P_1,P_2)$ where the weight of $P_1$ is $\mathrm{row}(A)$ and the weight of $P_2$ is $\mathrm{col}(A)$. We denote by $\lambda$ the common shape of $P_1$ and $P_2$. A theorem of Schensted  \cite{schensted1961longest} (cf.~Theorem 8 of Krattenthaler \cite{K} with $k=1$) tells us that the largest part of $\lambda$ (resp.~the number of parts in $\lambda$) equals the length of the longest  SE-chain (resp.~ne-chain) of $A$.
\begin{proposition}\label{prop:1steval}
Let $N \ge 1$. Given sequences $(a_1,\ldots,a_{\ell})$ and $(b_1,\ldots,b_{\ell})$ of nonnegative integers, let $\mu$ and $\widetilde{\mu}$ be the partitions with $a_j$ and $b_j$ parts of size $j$, respectively. The integral
\begin{equation}\label{eq:intsc}
	\int_{U(N)} \prod_{j=1}^{\ell} \Sc_j(U)^{a_j} \overline{\Sc_j(U)^{b_j}}\, dU
\end{equation}
is equal to the the number of $\ell(\mu) \times \ell(\widetilde{\mu})$ matrices $A$ with nonnegative integer entries such that $\mathrm{row}(A)=\mu$, $\mathrm{col}(A) = \widetilde{\mu}$ and the longest SE-chain in $A$ has length $\le N$. The integral
\begin{equation}\label{eq:inttc}
\int_{U(N)} \prod_{j=1}^{\ell} (\Tr \Sym^j(U))^{a_j} \overline{(\Tr \Sym^j(U))^{b_j}}\, dU
\end{equation}
is equal to the the number of $\ell(\mu) \times \ell(\widetilde{\mu})$ matrices $A$ with nonnegative integer entries such that $\mathrm{row}(A)=\mu$, $\mathrm{col}(A) = \widetilde{\mu}$ and the longest ne-chain in $A$ has length $\le N$.
\end{proposition}
Proposition~\ref{prop:1steval} generalizes \eqref{eq:DG} and Lemma~\ref{lem:trsym}. If $a_j=b_j=0$ for $j\ge 2$ and $a_1=b_1=n$, Proposition~\ref{prop:1steval} shows 
\begin{equation}
	\int_{U(N)} |\mathrm{Tr}(U)^n|^2\, dU
\end{equation}
is equal to the number of permutations in $S_n$ with longest increasing subsequence of length $\le N$, a result of Rains \cite{Rains}.
\begin{proof}[Proof of Proposition~\ref{prop:1steval}]
The proof of the evaluation of \eqref{eq:intsc} is almost the same as the proof of \eqref{eq:DG} in \cite{diaconis2004}, where $N \ge \max\{\sum_j j a_j, \sum_j jb_j\}$ was imposed. Instead of imposing this we invoke Schensted's theorem at the end of the proof:

As in the proof of \eqref{eq:DG} we may assume $\sum_j ja_j = \sum_j jb_j=n$ for some  $n\ge 0$. Since $\prod_{j=1}^{\ell} \Sc_j(U)^{a_j}=e_{\mu}(U)$ and $\prod_{j=1}^{\ell} \Sc_j(U)^{b_j}=e_{\widetilde{\mu}}(U)$, the expansions in \eqref{eq:exps} together with orthogonality of Schur functions \cite[Eq.~(22)]{diaconis2004} imply that \eqref{eq:intsc} equals
\[ \sum_{\substack{\lambda \vdash n \\ \ell(\lambda) \le N}} K_{\lambda',\mu}K_{\lambda',\widetilde{\mu}}=\sum_{\substack{\lambda \vdash n \\ \lambda_1 \le N}} K_{\lambda,\mu}K_{\lambda,\widetilde{\mu}},\]
i.e.~the number of pairs $(P,Q)$ of SSYTs where $P$ has weight $\mu$, $Q$ has weight $\widetilde{\mu}$, and $P$ and $Q$ have a common shape $\lambda \vdash n$ such that $\lambda_1 \le N$. By the RSK correspondence and Schensted's theorem \cite{schensted1961longest}, such pairs are in one-to-one correspondence with the matrices described in the first part of the proposition.

Now consider \eqref{eq:inttc}. As we saw in the proof of Lemma~\ref{lem:trsym}, \eqref{eq:inttc} is equal to the sum in \eqref{eq:partialsum}, i.e.~the number of pairs $(P,Q)$ of SSYTs where $P$ has weight $\mu$, $Q$ has weight $\widetilde{\mu}$, and $P$ and $Q$ have a common shape $\lambda$ such that $\lambda \vdash n$ and $\ell(\lambda) \le N$. By the RSK correspondence and Schensted's theorem \cite{schensted1961longest}, such pairs are in one-to-one correspondence with the matrices described in the second part of the proposition.
\end{proof} 
Let 
\[ I_k(n;N) := \int_{U(N)}|[u^n]\det(I+uU)^k|^2\, dU = \int_{U(N)}\left|\sum_{j_1+\ldots+j_k=n}\prod_{i=1}^{k} \Sc_{j_i}(U)\right|^2\, dU.\]
Summing \eqref{eq:intsc} over all sequences $(a_j)_j$ and $(b_j)_j$ of nonnegative integers such that $\sum_j ja_j=\sum_j b_j=n$, $\sum_j a_j = \sum_j b_j=k$ and applying the first part of Proposition~\ref{prop:1steval}, we obtain
\begin{cor}\label{cor:inter}
Let $n,k,N\ge 1$. Then $I_k(n;N)$ is equal to the number of $k\times k$ matrices with nonnegative integer entries whose sum is $n$, and their longest SE-chain has length $\le N$.
\end{cor}
The integral $I_k(n;N)$ was studied extensively in \cite{keating2018}. According to Theorem~1.5 of \cite{keating2018}, 
\[ I_k(n;N) = N^{k^2-1}(\gamma_k(n/N)+O_k(1/N)) \]
holds uniformly for $n,N\ge 1$, where $\gamma_k\colon \RR_{\ge 0} \to \RR_{\ge 0}$ is an explicit function  supported on $[0,k]$ \cite[Eq.~(1.12)]{keating2018}. Thus, in view of Corollary~\ref{cor:inter}, if we pick uniformly at random a $k\times k$ matrix with nonnegative integer entries that sum to $n$, the probability that its longest SE-chain has length $\le N$ can be shown to be
\[ \frac{I_k(n;N)}{I_k(n;n)} = \frac{\gamma_k(n/N)}{\gamma_k(1)(n/N)^{k^2-1}}+O_k(1/n).\]
Theorem~1.4 of \cite{keating2018} (cf.~\cite[Thm.~1.5]{Rodgers}) shows $I_k(n;N)$ equals the number of arrays $(x_{i,j})_{1\le i,j\le k}$ of nonnegative integers satisfying $x_{1,1}\le N$, $\sum_{i=1}^{k} x_{i,i} = n$, and $x_{i,j}$ is weakly decreasing when either $i$ or $j$ is fixed. We give a similar description of the integrals in Proposition~\ref{prop:1steval} using an idea of Rodgers \cite[p.~1270]{Rodgers}.
\begin{proposition}\label{prop:2ndeval}
	Let $N \ge 1$. Given sequences $(a_1,\ldots,a_{\ell})$ and $(b_1,\ldots,b_{\ell})$ of nonnegative integers, let $\mu$ and $\widetilde{\mu}$ be the partitions with $a_j$ and $b_j$ parts of size $j$, respectively. 	The integral \eqref{eq:intsc} (resp.~\eqref{eq:inttc}) equals the number of arrays $(x_{i,j})_{1\le i\le \ell(\mu),\,1\le j \le \ell(\widetilde{\mu})}$ of nonnegative integers satisfying each of the following conditions:
	\begin{enumerate}
		\item $x_{1,1} \leq N$ (resp.~$x_{i,i}=0$ for $N+1 \le i \le \min\{\ell(\mu),\ell(\widetilde{\mu})\}$),
		\item $x_{i,j}$ is weakly decreasing when either $i$ or $j$ is fixed,
		\item $\sum_{j-i = r-\ell(\mu)} x_{i,j}=\mu_1+\ldots+\mu_r$ for $1 \le r\le \ell(\mu)$ and $\sum_{i-j = s-\ell(\widetilde{\mu})} x_{i,j}=\widetilde{\mu}_1+\ldots+\widetilde{\mu}_r$ for $1 \le s \le \ell(\widetilde{\mu})$.
	\end{enumerate}
\end{proposition}
\begin{proof}
If $\sum_j j a_j \neq \sum_j j b_j$, i.e.~$\sum_i \mu_i \neq \sum_i \widetilde{\mu}_i$, then the integrals \eqref{eq:intsc} and \eqref{eq:inttc} vanish as in the proof of \eqref{eq:DG}. In this case there can be no arrays satisfying the third condition with $(r,s)=(\ell(\mu),\ell(\widetilde{\mu}))$, which is what we needed to show. From now on we assume that  $\sum_j j a_j = \sum_j j b_j=n$ for some $n$.

A Gelfand--Tsetlin pattern (or GT-pattern) of $k$ rows is a triangular array of nonnegative integers $(a_{i,j})_{1 \le i\le j \le k}$ with $a_{i,j}\le a_{i+1,j+1}\le a_{i,j+1}$ when $1\le i,j\le k-1$; its largest element is $a_{1,k}$. SSYTs with entries in $\{1,2,\ldots,k\}$ are in one-to-one correspondence, described in detail by Stanley \cite[pp.~313-314]{stanley}, with such arrays. SSYTs of shape $\lambda=(\lambda_1,\ldots,\lambda_m)$ and weight $(w_1,\ldots,w_r)$ ($r \le k$) correspond to GT-patterns of $k$ rows such that $a_{1,i}= \lambda_{k-i+1}$  and $\sum_{j\in [i,k]} a_{i,j}=w_1+w_2+\ldots+w_{k-i+1}$ for all $1 \le i \le k$ (here $w_i\equiv 0$ if $i>r$, $\lambda_i\equiv 0$ if $i>m$). In words, the first row of the pattern recovers the shape of the SSYT (by reversing the row), and the row sums recover the weight.

In the proof of Proposition~\ref{prop:1steval} we saw that \eqref{eq:intsc} (resp.~\eqref{eq:inttc}) equals the number of pairs $(P,Q)$ of SSYTs where $P$ has weight $\mu$, $Q$ has weight $\widetilde{\mu}$, and $P$ and $Q$ have the same shape. This common shape, let us call it $\lambda$, satisfies $\lambda \vdash n$ and $\lambda_1 \le N$ (resp.~$\ell(\lambda) \le N$). Necessarily $\ell(\lambda)\le \min\{\ell(\mu),\ell(\widetilde{\mu})\}$, otherwise there are no such pairs. We apply the above one-to-one correspondence with GT-patterns to obtain all pairs $((a_{i,j})_{1\le i \le j\le \ell(\mu)},(b_{i,j})_{1\le i \le j \le \ell(\widetilde{\mu})})$ of patterns, such that their (common) first row is $\lambda$ (in reverse order), and their respective row sums encode $\mu$ and $\widetilde{\mu}$. 

Assume without loss of generality that $\ell(\mu)\ge \ell(\widetilde{\mu})$. We make the following observation: since $\ell(\lambda) \le \ell(\widetilde{\mu})$, it follows that $a_{1,i}=0$ for all $1 \le i \le \ell(\mu)- \ell(\widetilde{\mu})$. Due to the property $a_{i,j}\le a_{i+1,j+1}\le a_{i,j+1}$, this forces $a_{i,j}=0$ for all $1 \le i \le j \le \ell(\mu)-\ell(\widetilde{\mu})$. 

Next we define an array $(x_{i,j})_{1 \le i\le \ell(\mu),\, 1\le j \le \ell(\widetilde{\mu})}$, satisfying the three conditions in the proposition, by $x_{i,j}:=a_{1+i-j,\ell(\mu)+1-j}$ if $i \ge j$ and $x_{i,j}:=b_{1+j-i,\ell(\widetilde{\mu})+1-i}$ if $i \le j$, giving the desired result.

In the secular coefficient case, the condition $\lambda_1 \le N$ is encoded by the inequalities $a_{1,\ell(\mu)}=b_{1,\ell(\widetilde{\mu})}\le N$ which become $x_{1,1}\le N$. In the symmetric traces case, the condition $\ell(\lambda)\le N$ is encoded by the relations $a_{1,i}=0$ for $1 \le i \le \ell(\mu)-N$ and $b_{1,i}=0$ for $1 \le i \le \ell(\widetilde{\mu})-N$ (recall $\lambda$ -- in reverse order -- is the first row of the patterns $(a_{i,j})_{1\le i \le j \le \ell(\mu)}$ and  $(b_{i,j})_{1\le i \le j \le \ell(\widetilde{\mu})}$), which become $x_{i,i}=0$ for $N+1 \le i \le \ell(\widetilde{\mu})$.
\end{proof}
Summing Proposition~\ref{prop:2ndeval} over all sequences $(a_j)_j$ and $(b_j)_j$ of $k$ nonnegative integers such that $\sum_j ja_j=\sum_j jb_j = n$ and $\sum_j a_j =\sum_j b_j= k$, we recover Theorem~1.4 of \cite{keating2018}.
\section{Number-theoretic connections}\label{sec:NT}
\subsection{A polynomial analogue of \texorpdfstring{$d_{\lambda}$}{dlambda}}
Let $\lambda \vdash n$. Over $\mathbb{F}_q[T]$, the polynomial ring over the finite field of $q$ elements, it is straightforward to define an analogue of $d_{\lambda}$. Letting $\mathcal{M}_{n,q} \subseteq \mathbb{F}_q[T]$ be the subset of monic polynomials of degree $n$, we set
\[ d_{\lambda,q}(f):=\sum_{\substack{f_1 \cdots f_{\ell(\lambda)}=f \\ \forall i:\, \deg (f_i) = \lambda_i,\, f_i \text{ monic}}} 1. \]
If $\mu$, $\widetilde{\mu}$ are partitions of $n$ then we claim
\begin{equation}\label{eq:largeqlim}
 \lim_{q\to \infty}\frac{1}{q^n} \sum_{f \in \mathcal{M}_{n,q}} d_{\mu,q}(f) d_{\widetilde{\mu},q}(f) = \frac{1}{|S_n|} \sum_{ \pi \in S_n} d_{\mu}(\pi) d_{\widetilde{\mu}}(\pi)=N_{\mu,\widetilde{\mu}}.
\end{equation}
The second equality in \eqref{eq:largeqlim} is just Proposition~\ref{prop:avoid knuth}. The first equality in \eqref{eq:largeqlim} has to do with a general principle and applies to a general class of functions\footnote{Let $\alpha \colon S_n \to \mathbb{C}$ be class function. Let $\alpha_q \colon \mathcal{M}_{n,q} \to \mathbb{C}$ be a function with the following property: if $f$ is squarefree then $\alpha_q(f)=\alpha(\pi)$ whenever the multiset of degrees of the primes in the factorization of $f$ matches with the multiset of cycle lengths of the cycles in the cycle decomposition of $\pi$. Then $\sum_{f \in \mathcal{M}_{n,q}} \alpha_q(f)/q^n = \sum_{\pi \in S_n} \alpha(\pi)/n! + O_{n,\max|\alpha_q|}(1/q)$. This follows e.g.~from Cohen \cite{Cohen1970}.} However, the fact that the left-hand side of \eqref{eq:largeqlim} equals to the right-hand side can be established directly, as we sketch now. By definition, $\sum_{f \in \mathcal{M}_{n,q}} d_{\mu,q}(f) d_{\widetilde{\mu},q}(f)$ counts solutions to the equation
\[ f_1 f_2 \cdots f_{\ell(\mu)} = g_1 g_2 \cdots g_{\ell(\widetilde{\mu})}\]
where $\deg f_i = \mu_i$ and $\deg g_j = \widetilde{\mu}_j$. We explain how to count such solutions. We can associate with any pair of tuples of solutions $(f_i)_i$ and $(g_j)_j$ a \emph{gcd matrix}
\[ h_{i,j} := \gcd(g_i,f_j),\]
analogous to the matrix constructed in the proof of Proposition~\ref{prop:avoid knuth}.
At least if the $(g_i)_i$ are pairwise coprime and so are $(f_j)_j$, we may reconstruct them from $h_{i,j}$ via $g_i = \prod_{j}h_{i,j}$ and $h_j = \prod_{i} h_{i,j}$. Fortunately, in the large-$q$ limit we can impose these coprimality conditions and only incur an error of $o_{q \to \infty}(q^{n})$ (the details are left to the reader).

Given a gcd matrix $(h_{i,j})_{i,j}$ (coming from pairwise coprime $(g_i)_i$ and pairwise coprime $(f_j)_j$) we can further form a \emph{degree matrix} $(\deg h_{i,j})_{i,j}$ of the same size, which is counted by $N_{\mu,\widetilde{\mu}}$. For each degree matrix $(d_{i,j})_{i,j}$ counted by $N_{\mu,\widetilde{\mu}}$ we need to count the number of gcd matrices corresponding to it, that is, monic $h_{i,j}$ with $\deg h_{i,j}=d_{i,j}$ and $h_{i,j}$ being pairwise coprime (a consequence of the pairwise coprimality of $(g_i)_i$ and $(f_j)_j$). As already mentioned, in the large-$q$ limit these coprimality conditions do not affect asymptotics, and the number of such $h_{i,j}$ is $q^{\sum_{i,j} d_{i,j}}(1+o_{q \to \infty}(1))=q^{n}(1+o_{q \to \infty}(1))$ (so, always asymptotic to $q^n$, regardless of the specific $(d_{i,j})_{i,j}$). All in all, we obtain that the left-hand side of \eqref{eq:largeqlim} is equal to the right-hand side, without using the first equality in \eqref{eq:largeqlim}.

This idea can be made to work for fixed $q$ as well using a clever construction of Vaughan and Wooley \cite[\S8]{Vaughan1995}, which can be used to show   that in fact $\sum_{f \in \mathcal{M}_{n,q}} d_{\mu,q}(f) d_{\widetilde{\mu},q}(f)$ is a polynomial in $q$ of degree $n$ and leading coefficient $N_{\mu,\widetilde{\mu}}$, but we do not give details. They construct matrices which take into account common factors. They used an inductive process to associate with a solution of $m_1 \cdots m_{\ell(\mu)} = n_1 \cdots n_{\ell(\widetilde{\mu})}$ a $\ell(\mu) \times \ell(\widetilde{\mu})$ matrix $(a_{i,j})_{i,j}$ such that $m_i=\prod_{i=1}^{\ell(\widetilde{\mu})}a_{r,i}$ and $n_i=\prod_{i=1}^{\ell(\mu)}a_{i,r}$; this also works with polynomials instead of integers. The process goes by letting $a_{1,1}=\gcd(m_1,n_1)$ and then defining, using induction on $i+j$, $a_{i,j}=\gcd(m_i/\prod_{\ell<j}a_{i,\ell},n_j/\prod_{\ell<i}a_{\ell,j})$.
The above process was discovered independently by Granville and Soundararajan in the proof of \cite[Thm.~4]{Granville2001}.
We explain how these gcd matrices arose in their work because it is not to difficult to relate it to matrix integrals. They were interested in the order of magnitude of moments of character sums:
\[ M_k(x,q) = \frac{1}{\phi(q)}\sum_{\chi \bmod q} \left|\sum_{n \le x} \chi(n)\right|^{2k}.\]
Here the sum is over all Dirichlet characters modulo $q$. Using orthogonality of characters, we have
\[ M_k(x,q) = \#\{ n_1 n_2 \cdots n_k \equiv m_1 m_2 \cdots m_k \bmod q,\, \forall  i:\, n_i,m_i \le x,\, (n_im_i,q)=1\}.\]
If $x^k \le q$, this reduces to
\[ M_k(x,q) = \#\{ n_1 n_2 \cdots n_k = m_1 m_2 \cdots m_k,\, \forall  i:\,  n_i,m_i \le x,\, (n_im_i,q)=1\}.\]
They were also interested in moments of sums of the (Steinhaus) random multiplicative function, which we denote by $\alpha$ and whose definition we now recall. It is a random \textit{completely} multiplicative function ($\alpha(nm)=\alpha(n)\alpha(m)$ for all $n,m\ge 1$), chosen in such a way that $(\alpha(p))_p$ ($p$ prime) are i.i.d.~random variables taking values uniformly on $\{z \in \mathbb{C}: |z|=1\}$. Then we have the orthogonality relation
\begin{equation}\label{eq:Steinhaus}
\mathbb{E} \alpha(n) \overline{\alpha}(m) = \delta_{nm}
\end{equation}
and similarly, 
\[ M_k(x) := \mathbb{E}\left|\sum_{n \le x} \alpha(n)\right|^{2k}= \#\{ n_1 n_2 \cdots n_k = m_1 m_2 \cdots m_k,\, \forall  i:\, n_i,m_i \le x\}.\]
The gcd matrices one associates with the solutions counted by $M_k(x,q)$ (if $x^k \le q$) and $M_k(x)$ can be counted and in turn lead to bounds on $M_k(x,q)$ and $M_k(x)$ as in \cite[Thms.~4.1--4.2]{Granville2001}. We refer the reader to Heap and Lindqvist \cite{Heap2016} for asymptotics results for $M_k(x,q)$ ($x^k \le q$) and $M_k(x)$ (cf.~Harper, Nikeghbali and Radziwi\l\l  \, \cite{Harper2015}). See also Harper \cite{Harper2019,Harper2020} for estimates on $M_k(x)$ in a wide range of $k\ge 0$, including noninteger $k$.

In function fields, the connection between character sums and secular coefficients is natural. If $\chi$ is a nonprincipal Dirichlet character modulo $Q$, we can form its Dirichlet $L$-function
\[ L(u,\chi) = \sum_{f \text{ monic}} \chi(f)u^{\deg f},\]
which is a polynomial whose $n$th coefficient is \emph{exactly} a character sum. By Weil's Riemann hypothesis, we can see that
\begin{equation}\label{eq:secular}
\sum_{f \in \mathcal{M}_{n,q}} \chi(f) \ll_{n,\deg Q} q^{\frac{n}{2}}.
\end{equation}
If $\chi$ is odd (i.e.~$\chi$ is not trivial on $\FF_q^{\times}$) and primitive we have $\deg L(u,\chi)= \deg Q - 1$ and the zeros of $L$ all lie on $|u|=q^{-1/2}$ so that we can write $L(u,\chi)$ as a characteristic polynomial of a scaled unitary matrix:
\[ L(u,\chi) = \det(I_{\deg Q-1} - u\sqrt{q} \Theta_{\chi}),\, \Theta_{\chi} \in U(\deg Q-1).\]
Comparing coefficients, we find that
\begin{equation}\label{eq:sec}
\Sc_n(\Theta_{\chi}) = \frac{(-1)^n}{q^{\frac{n}{2}}} \sum_{f \in \mathcal{M}_{n,q}} \chi(f).
\end{equation}
If $nk \le \deg Q$ then orthogonality relations show
\begin{equation}\label{eq:orthchar} \frac{1}{\phi(Q)} \sum_{\chi \bmod Q} \left|\sum_{f \in \mathcal{M}_{n,q}} \chi(f) \right|^{2k} =  \#\{ f_1 f_2 \cdots f_k = g_1 g_2 \cdots g_k, \, \forall i:\,\deg f_i = \deg g_i = n,\, (f_i g_i,Q)=1\}.
\end{equation}
Since in the large-$q$ limit a random polynomial will be coprime to $Q$ with probability approaching $1$, we can deduce from \eqref{eq:orthchar} (using gcd matrices as used in establishing \eqref{eq:largeqlim}) that
\begin{equation}\label{eq:magicfqt}
\frac{1}{\phi(Q)} \sum_{\chi \bmod Q} \left|\sum_{f \in \mathcal{M}_{n,q}} \chi(f) \right|^{2k} =(1+o_{q \to \infty}(1)) q^{nk} N_{(n^k),(n^k)}	
\end{equation}
holds as $q \to \infty$, where we assume $nk \le \deg Q$ and that $n$, $k$ and $\deg Q$ are fixed. Here $(n^k)$ stands for the partition of $nk$ consisting of $n$ repeated $k$ times. The term $o_{q \to \infty}(1)$ goes to $0$ as $q \to \infty$ and may depend on the fixed parameters. We can also package \eqref{eq:orthchar} as
\begin{equation}\label{eq:pack}
\frac{1}{\phi(Q)} \sum_{\chi \bmod Q} \left|\sum_{f \in \mathcal{M}_{n,q}} \chi(f) \right|^{2k} = \mathbb{E} \left|\sum_{\substack{f \in \mathcal{M}_{n,q}\\(f,Q)=1}} \alpha(f)\right|^{2k}
\end{equation}
where now $\alpha$ is a Steinhaus random multiplicative function in $\mathbb{F}_q[T]$ which satisfies orthogonality relations similar to \eqref{eq:Steinhaus}, and $nk \le  \deg Q$.

In the large-$q$ limit almost all characters are primitive and odd \cite[Eq.~(3.25)]{KR} and we obtain from \eqref{eq:secular} and \eqref{eq:sec} that
\begin{equation}\label{eq:comp}
\frac{1}{\phi(Q)} \sum_{\chi \bmod Q} \left|\sum_{f \in \mathcal{M}_{n,q}} \chi(f) \right|^{2k} = q^{nk} \left( \mathbb{E}_{\substack{\chi\text{ odd,}\\ \text{primitive mod }Q}} \left|\Sc_{n}(\Theta_{\chi})\right|^{2k} + o_{q \to \infty}(1)\right)
\end{equation}
if $nk \le \deg Q -1$ (to handle the contribution of the principal character). Comparing \eqref{eq:comp} and \eqref{eq:magicfqt} we find that
\begin{equation}\label{eq:comp2}
\lim_{q \to \infty}\mathbb{E}_{\substack{\chi\text{ odd,}\\ \text{primitive mod }Q}} \left|\Sc_{n}(\Theta_{\chi})\right|^{2k} = N_{(n^k),(n^k)} = \int_{U(\deg Q - 1)} \left| \Sc_{n}(U)\right|^{2k} dU
\end{equation}
where the last equality is \eqref{eq:DG}. Here $n$, $k$ and $\deg Q$ are fixed and satisfy $nk \le \deg Q-1$. The fact that the left-hand side of \eqref{eq:comp2} converges to the right-hand side is essentially a special case of a deep equidistribution theorem of Katz \cite{Katz2013}. Katz proved that -- at least for squarefree $Q$ -- the ensemble $( \Theta_{\chi})_{\chi\text{ odd, primitive mod }Q}$ \textit{equidistributes} in $U(\deg Q-1)$ as $q \to \infty$. In other words, the average of a continuous function $f\colon U(\deg Q-1)\to \mathbb{C}$ over the finite ensemble $(\Theta_{\chi})_{\chi\text{ odd, primitive mod }Q}$ converges, as $q \to \infty$, to an average of $f$ over the full group $U(\deg Q-1)$.\footnote{Strictly speaking, Katz requires $f$ to be invariant under scalars: $f(cU)=f(U)$ for $|c|=1$.} For certain test functions $f$, Katz's result can be proved elementarily, and in fact his proof proceeds by first (without using algebraic geometry) establishing it for a particular family of functions. The elementary proof of \eqref{eq:comp2}, which corresponds to $f(U)=|\Sc_n(U)|^{2k}$, certainly \textit{does not} generalize to general functions.

What about traces of symmetric powers? These arise when twisting the M\"obius function by a character. Given a Dirichlet character $\chi$ modulo $Q$, we have
\begin{equation}\label{eq:recip}
 \frac{1}{L(u,\chi)} = \sum_{f \text{ monic}} \mu(f)\chi(f)u^{\deg f}.
\end{equation}
The $n$th coefficient of this rational function is
\begin{equation}\label{eq:RHrecip}
\sum_{f \in \mathcal{M}_{n,q}} \mu(f)\chi(f) \ll_{n, \deg Q} q^{\frac{n}{2}}
\end{equation}
by Weil's Riemann hypothesis. If $\chi$ is odd and primitive we have
\begin{equation}
\frac{1}{L(u,\chi)} = \frac{1}{\det(I_{\deg Q-1}-u\sqrt{q}\Theta_{\chi})} = \sum_{n=0}^{\infty} q^{\frac{n}{2}}u^n \Tr\Sym^n (\Theta_{\chi}),
\end{equation}
i.e.
\begin{equation}\label{eq:RecipIden}
	\sum_{f \in \mathcal{M}_{n,q}} \mu(f)\chi(f)=q^{\frac{n}{2}} \Tr\Sym^n (\Theta_{\chi})
\end{equation}
for all $n$. In the large-$q$ limit almost all characters are primitive and odd and we obtain from \eqref{eq:recip}, \eqref{eq:RHrecip} and \eqref{eq:RecipIden} that
\[  \frac{1}{\phi(Q)} \sum_{\chi \bmod Q} \left|\sum_{f \in \mathcal{M}_{n,q}} \mu(f)\chi(f) \right|^{2k} =  q^{nk} \left(  \mathbb{E}_{\substack{\chi\text{ odd,}\\ \text{primitive mod }Q}} \left|\Tr\Sym^n(\Theta_{\chi})\right|^{2k} + o_{q \to \infty}(1)\right)\]
if $k\le \deg Q -1$ (to handle the contribution of the principal character). In particular, by using Katz's equidistribution result  \cite[Thm.]{Katz2013},
\begin{equation}\label{eq:KatzUse}
\frac{1}{\phi(Q)} \sum_{\chi \bmod Q} \left|\sum_{f \in \mathcal{M}_{n,q}} \mu(f)\chi(f) \right|^{2k} \sim q^{nk}\int_{U(\deg Q-1)} 
 \left|\Tr\Sym^n(U)\right|^{2k} = q^{nk}N_{(n^k),(n^k)}
\end{equation}
holds as $q \to \infty$ if $k\le \deg Q-1$, at least if $Q$ is squarefree (a condition required by Katz). In the second equality in \eqref{eq:KatzUse} we used Lemma~\ref{lem:trsym}, which is the reason for the range $k \le \deg Q-1$.  

In the restricted range $nk \le \deg Q$ we can mimic \eqref{eq:orthchar}--\eqref{eq:pack} and obtain
\begin{equation}\label{eq:mimic}
\frac{1}{\phi(Q)} \sum_{\chi \bmod Q} \left|\sum_{f \in \mathcal{M}_{n,q}} \mu(f)\chi(f) \right|^{2k} = \mathbb{E} \left|\sum_{\substack{f \in \mathcal{M}_{n,q}\\(f,Q)=1}} \mu^2(f)\alpha(f) \right|^{2k}.
\end{equation}
The asymptotic relation \eqref{eq:KatzUse} can be shown to imply that if we replace equality in \eqref{eq:mimic} with an asymptotic then it continues to hold in the much wider range $k \le \deg Q - 1$, at least if one takes a large-$q$ limit and assumes $Q$ is squarefree. We believe this should hold without taking the large-$q$ limit and for all $Q$, and we suggest the following conjecture in integers.
\begin{conj}\label{conj:mob}
Suppose $q$ and $x$ tend to $\infty$ and that $k < \log q$. We have 
\begin{align}\label{eq:conj}\frac{1}{\phi(q)} \sum_{\chi \bmod q} \left|\sum_{n \le x} \mu(n)\chi(n)\right|^{2k} &\sim \mathbb{E} \left| \sum_{\substack{n \le x\\(n,q)=1}} \mu^2(n)\alpha(n)\right|^{2k},\\
\frac{1}{\phi(q)} \sum_{\chi \bmod q} \left|\sum_{n \le x} \lambda(n)\chi(n)\right|^{2k} &\sim \mathbb{E} \left| \sum_{\substack{n \le x\\(n,q)=1}} \alpha(n)\right|^{2k}.
\end{align}
Here $k$ is allowed to vary with $x$ and $q$, and $\lambda$ is the Liouville function $\lambda(n)=(-1)^{\sum_{p,\, k \ge 1:\, p^k \mid n} 1}$.
\end{conj}
This should be contrasted with the more modest range $x^k\le q$ that occurs in the same problem but without an appearance of $\mu$ or $\lambda$:
\begin{equation}\label{eq:momchi} \frac{1}{\phi(q)} \sum_{\chi \bmod q} \left|\sum_{n \le x} \chi(n)\right|^{2k} = \mathbb{E} \left| \sum_{\substack{n \le x\\(n,q)=1}} \alpha(n)\right|^{2k}.
\end{equation}
\begin{remark}
The range $x^k \le q$ is essentially optimal (up to $x^{o(1)}$), even if we replace equality with an asymptotic. In the current formulation this is trivial, since the principal character contributes $x^{2k}/\phi(q)$ to the left-hand side of \eqref{eq:momchi}. If one removes the principal character, this still should be optimal. We do not attempt to demonstrate this here, but make two comments: 1) if $x$ is an integer divisible by $q$ then $\sum_{n\le x} \chi(n)$ vanishes for any nonprincipal $\chi$, showing the left-hand side of \eqref{eq:momchi} vanishes for $x=q$ if we remove the principal character, 2) the optimality claim is related to $\int_{U(N)} |\Sc_n(U)|^{2k}\,dU \sim  N_{(n^k),(n^k)}$ failing to hold if $N \le nk(1-\varepsilon)$ and $n \to \infty$, which we expect can be demonstrated using Proposition~\ref{prop:2ndeval} and the tools in \cite{keating2018}. 
\end{remark}
Already for $k=1$ Conjecture~\ref{conj:mob} is a very difficult open problem, related to the variance of $\mu$ and $\lambda$ in arithmetic progressions, see e.g.~\cite{Keating2016}.

We view Conjecture~\ref{conj:mob} as a manifestation of M\"obius randomness. One interpretation of it is that the Steinhaus random multiplicative function $\alpha$ (resp.~$\alpha \cdot \mu^2$) is a good model for $\lambda$ (resp.~$\mu$) twisted by a random Dirichlet character $\chi$ modulo $q$ ($q \to \infty$)\footnote{Of course, $\alpha$ is never zero while $\chi$ vanishes at primes dividing $q$. This can be remedied by multiplying $\alpha$ by the principal character modulo $q$, or alternatively by working with $q$ that is a prime.}. It certainly models $\lambda$ times a random character better than it models just a random character, at least from the point of view of moments.

Let 
\begin{align}
\mu_{x,k}(n) &:= \sum_{n_1 n_2 \cdots n_k = n, \, \forall i: \, n_i \le x} \prod_{i=1}^{k} \mu(n_i),\\
\lambda_{x,k}(n) &:= \sum_{n_1 n_2 \cdots n_k = n, \, \forall i: \, n_i \le x} \prod_{i=1}^{k} \lambda(n_i)=\lambda(n)\sum_{n_1 n_2 \cdots n_k = n, \, \forall i: \, n_i \le x} 1.
\end{align}
From orthogonality of characters we have the identities
\begin{equation}\label{eq:iden}
\frac{1}{\phi(q)} \sum_{\chi \bmod q} \left|\sum_{n \le x}\mu(n)\chi(n)\right|^{2k} =\frac{1}{\phi(q)} \sum_{\chi \bmod q} \left|\sum_{n \le x^k}\mu_{x,k}(n)\chi(n)\right|^{2}=  \sum_{\substack{n,m \le x^k\\ n \equiv m \bmod q \\ (nm,q)=1}} \mu_{x,k}(n) \mu_{x,k}(m).
\end{equation}
and similarly
\begin{equation}\label{eq:iden2}
	\frac{1}{\phi(q)} \sum_{\chi \bmod q} \left|\sum_{n \le x}\lambda(n)\chi(n)\right|^{2k} =\frac{1}{\phi(q)} \sum_{\chi \bmod q} \left|\sum_{n \le x^k}\lambda_{x,k}(n)\chi(n)\right|^{2}=  \sum_{\substack{n,m \le x^k\\ n \equiv m \bmod q \\ (nm,q)=1}} \lambda_{x,k}(n) \lambda_{x,k}(m).
\end{equation}
Observe that the diagonal terms in \eqref{eq:iden} are $\sum_{n \le x^k, \, (n,q)=1} \mu_{x,k}(n)^2$, and this sum satisfies (using \eqref{eq:Steinhaus})
\[\sum_{n \le x^k, \, (n,q)=1} \mu_{x,k}(n)^2= \mathbb{E} \left| \sum_{\substack{n \le x\\(n,q)=1}} \mu^2(n) \alpha(n)\right|^{2k}.\]
Hence, Conjecture~\ref{conj:mob} says that the diagonal terms $n=m$ give (asymptotically) the main contribution to the $2k$th moment in \eqref{eq:iden}, as long as $k< \log q$ and $\min\{q,x\} \to \infty$. Similarly, the same applies to \eqref{eq:iden2}, according to the conjecture.
\subsection*{Acknowledgements}
We thank Jon Keating and Brad Rodgers for useful comments and suggestions and Zachary Chase for a typographical correction. This project has received funding from the European Research Council (ERC) under the European Union's Horizon 2020 research and innovation programme (grant agreement No 851318).

\subsection*{Data availability}
Data availability is not applicable to this article as no new data were created or analysed in this study.
\bibliographystyle{abbrv}
\bibliography{references}

\begin{thebibliography}{10}

\bibitem{Baxter1961}
G.~Baxter.
\newblock Polynomials defined by a difference system.
\newblock {\em J. Math. Anal. Appl.}, 2:223--263, 1961.

\bibitem{bumplie}
D.~Bump.
\newblock {\em Lie groups}, volume 225 of {\em Graduate Texts in Mathematics}.
\newblock Springer-Verlag, New York, 2004.

\bibitem{bump}
D.~Bump and A.~Gamburd.
\newblock On the averages of characteristic polynomials from classical groups.
\newblock {\em Comm. Math. Phys.}, 265(1):227--274, 2006.

\bibitem{Cohen1970}
S.~D. Cohen.
\newblock The distribution of polynomials over finite fields.
\newblock {\em Acta Arith.}, 17:255--271, 1970.

\bibitem{conrey2006}
B.~Conrey and A.~Gamburd.
\newblock Pseudomoments of the {R}iemann zeta-function and pseudomagic squares.
\newblock {\em J. Number Theory}, 117(2):263--278, 2006.

\bibitem{conrey2005howe}
J.~Conrey, D.~Farmer, and M.~Zirnbauer.
\newblock Howe pairs, supersymmetry, and ratios of random characteristic
  polynomials for the unitary groups {U}({N}).
\newblock {\em arXiv preprint math-ph/0511024}, 2005.

\bibitem{diaconis2004}
P.~Diaconis and A.~Gamburd.
\newblock Random matrices, magic squares and matching polynomials.
\newblock {\em Electron. J. Combin.}, 11(2):Research Paper 2, 26, 2004/06.

\bibitem{diaconis1994}
P.~Diaconis and M.~Shahshahani.
\newblock On the eigenvalues of random matrices.
\newblock volume 31A, pages 49--62. 1994.
\newblock Studies in applied probability.

\bibitem{EFG}
S.~Eberhard, K.~Ford, and B.~Green.
\newblock Permutations fixing a {{\(k\)}}-set.
\newblock {\em Int. Math. Res. Not.}, 2016(21):6713--6731, 2016.

\bibitem{Granville2001}
A.~Granville and K.~Soundararajan.
\newblock Large character sums.
\newblock {\em J. Amer. Math. Soc.}, 14(2):365--397, 2001.

\bibitem{haake1}
F.~Haake, M.~Ku\'{s}, H.-J. Sommers, H.~Schomerus, and K.~\.{Z}yczkowski.
\newblock Secular determinants of random unitary matrices.
\newblock {\em J. Phys. A}, 29(13):3641--3658, 1996.

\bibitem{Harper2019}
A.~J. Harper.
\newblock Moments of random multiplicative functions, {II}: {H}igh moments.
\newblock {\em Algebra Number Theory}, 13(10):2277--2321, 2019.

\bibitem{Harper2020}
A.~J. Harper.
\newblock Moments of random multiplicative functions, {I}: {L}ow moments,
  better than squareroot cancellation, and critical multiplicative chaos.
\newblock {\em Forum Math. Pi}, 8:e1, 95, 2020.

\bibitem{Harper2015}
A.~J. Harper, A.~Nikeghbali, and M.~Radziwi\l~\l.
\newblock A note on {H}elson's conjecture on moments of random multiplicative
  functions.
\newblock In {\em Analytic number theory}, pages 145--169. Springer, Cham,
  2015.

\bibitem{Heap2016}
W.~P. Heap and S.~Lindqvist.
\newblock Moments of random multiplicative functions and truncated
  characteristic polynomials.
\newblock {\em Q. J. Math.}, 67(4):683--714, 2016.

\bibitem{johansson1997}
K.~Johansson.
\newblock On random matrices from the compact classical groups.
\newblock {\em Ann. of Math. (2)}, 145(3):519--545, 1997.

\bibitem{Katz2013}
N.~M. Katz.
\newblock On a question of {K}eating and {R}udnick about primitive {D}irichlet
  characters with squarefree conductor.
\newblock {\em Int. Math. Res. Not. IMRN}, (14):3221--3249, 2013.

\bibitem{Keating2016}
J.~Keating and Z.~Rudnick.
\newblock Squarefree polynomials and {M}\"{o}bius values in short intervals and
  arithmetic progressions.
\newblock {\em Algebra Number Theory}, 10(2):375--420, 2016.

\bibitem{keating2018}
J.~P. Keating, B.~Rodgers, E.~Roditty-Gershon, and Z.~Rudnick.
\newblock Sums of divisor functions in {$\mathbb{F}_q[t]$} and matrix
  integrals.
\newblock {\em Math. Z.}, 288(1-2):167--198, 2018.

\bibitem{KR}
J.~P. Keating and Z.~Rudnick.
\newblock The variance of the number of prime polynomials in short intervals
  and in residue classes.
\newblock {\em Int. Math. Res. Not. IMRN}, (1):259--288, 2014.

\bibitem{K}
C.~Krattenthaler.
\newblock Growth diagrams, and increasing and decreasing chains in fillings of
  {Ferrers} shapes.
\newblock {\em Adv. Appl. Math.}, 37(3):404--431, 2006.

\bibitem{Najnudel}
J.~Najnudel, E.~Paquette, and N.~Simm.
\newblock Secular coefficients and the holomorphic multiplicative chaos.
\newblock {\em Ann. Probab.}, 51(4):1193--1248, 2023.

\bibitem{Rains}
E.~M. Rains.
\newblock Increasing subsequences and the classical groups.
\newblock {\em Electron. J. Comb.}, 5(1):research paper r12, 9, 1998.

\bibitem{Rodgers}
B.~Rodgers.
\newblock Arithmetic functions in short intervals and the symmetric group.
\newblock {\em Algebra Number Theory}, 12(5):1243--1279, 2018.

\bibitem{schensted1961longest}
C.~Schensted.
\newblock Longest increasing and decreasing subsequences.
\newblock {\em Canadian Journal of mathematics}, 13:179--191, 1961.

\bibitem{sommers2}
H.-J. Sommers, F.~Haake, and J.~Weber.
\newblock Joint densities of secular coefficients for unitary matrices.
\newblock {\em J. Phys. A}, 31(19):4395--4401, 1998.

\bibitem{stanley}
R.~P. Stanley.
\newblock {\em Enumerative combinatorics. {V}ol. 2}, volume~62 of {\em
  Cambridge Studies in Advanced Mathematics}.
\newblock Cambridge University Press, Cambridge, 1999.
\newblock With a foreword by Gian-Carlo Rota and appendix 1 by Sergey Fomin.

\bibitem{Vaughan1995}
R.~C. Vaughan and T.~D. Wooley.
\newblock On a certain nonary cubic form and related equations.
\newblock {\em Duke Math. J.}, 80(3):669--735, 1995.

\end{thebibliography}

\Addresses
\end{document}